\documentclass[letterpaper,10 pt, conference]{IEEEconf}
\IEEEoverridecommandlockouts

\usepackage{amsmath, amssymb}
\usepackage{mathtools}
\usepackage{tikz}

\DeclarePairedDelimiter\lr{\lparen}{\rparen}

\usetikzlibrary{arrows,automata, patterns, calc,decorations.pathmorphing,decorations.markings}

\newcommand{\bbR}{\mathbb{R}}
\newcommand{\bbC}{\mathbb{C}}

\newcommand{\calC}{\mathcal{C}}

\newcommand{\btab}{\begin{center}\def\arraystretch{1.5}\begin{tabular}}
        \newcommand{\etab}{\end{tabular}\end{center}}
\newcommand{\bbm}{\begin{bmatrix*}}
    \newcommand{\ebm}{\end{bmatrix*}}
\newcommand{\bvm}{\begin{vmatrix*}}
    \newcommand{\evm}{\end{vmatrix*}}

\newcommand{\inv}{^{-1}}
\renewcommand{\t}{^\top}

\newcommand{\set}[2]{\left\{ #1 ~\left|~ \vphantom{#1} #2 \right. \right\}}

\newcommand{\qand}{\quad\text{and}\quad}

\newcommand{\ddt}[1]{\tfrac{\d^{#1}}{\d t^{#1}}}
\newcommand{\pddt}{\lr*{\ddt{}}}
\renewcommand{\d}{\textup{d}}

\newcommand{\bstrut}{\rule[-1.4ex]{0pt}{0pt}}

\newcommand{\cinf}[1]{\calC^\infty_{#1}}
\newcommand{\sys}{{\Sigma}}
\newcommand{\ass}{\text{\upshape A}}
\newcommand{\env}{\text{\upshape E}}
\newcommand{\gar}{{\Gamma}}
\newcommand{\con}{\mathcal{C}}
\newcommand{\contract}[1]{\con_{#1} = (\ass_{#1}, \gar_{#1})}

\newcommand{\simby}{\preccurlyeq}

\newcommand{\meet}{\wedge}
\newcommand{\join}{\vee}

\newcommand{\B}[1]{\mathfrak{B}\lr{#1}}

\newcommand{\Bi}[1]{\mathfrak{B}_{\textup{i}}\lr{#1}}

\newcommand{\Bo}[1]{\mathfrak{B}_{\textup{o}}\lr{#1}}

\renewcommand{\epsilon}{\varepsilon}

\newtheorem{lemma}{Lemma}

\newtheorem{remark}{Remark}
\newtheorem{definition}{Definition}

\newtheorem{theorem}{Theorem}

\title{Behavioural contracts for linear dynamical systems:\\ input assumptions and output guarantees}
\author{B. M. Shali, A. J. van der Schaft, B. Besselink\thanks{The authors are with the Jan C. Willems Center for Systems and Control, and the Bernoulli Institute for Mathematics, Computer Science, and Artificial Intelligence, University of Groningen, Groningen, The Netherlands; Email: {\emph{b.m.shali@rug.nl}}; \emph{a.j.van.der.schaft@rug.nl}; \emph{b.besselink@rug.nl}.}}
\begin{document}
	\maketitle

    \begin{abstract}
        We introduce contracts for linear dynamical systems with inputs and outputs. Contracts are used to express formal specifications on the dynamic behaviour of such systems through two aspects: assumptions and guarantees. The assumptions are a linear system that captures the available knowledge about the dynamic behaviour of the environment in which the system is supposed to operate. The guarantees are a linear system that captures the required dynamic behaviour of the system when interconnected with its environment. In addition to contracts, we also define and characterize notions of contract refinement and contract conjunction. Contract refinement allows one to determine if a contract expresses a stricter specifications than another contract. On the other hand, contract conjunction allows one to combine multiple contracts into a single contract that fuses the specifications they express.
    \end{abstract}

    \vspace{-3mm}

    \section{Introduction}

    Modern engineering systems, such as smart grids and intelligent transportation systems, often comprise a large number of interconnected physical components that are modelled as continuous dynamical systems. Specifications on such components typically belong to one of two categories: dissipativity or set-invariance. Dissipativity theory \cite{willems1972} provides an elegant unifying framework that captures requirements such as stability, passivity or performance, while set-invariance techniques \cite{blanchini1999} are {\color{black} natural candidates for expressing} safety requirements. However, the growing complexity of modern engineering systems necessitates a theory of specifications that goes beyond dissipativity and set-invariance. {\color{black} Namely, these frameworks generally do not permit \emph{dynamic} specifications: the supply rates that are central to dissipativity theory are static, as are typical invariant sets. Moreover, when employed as specifications on components of large-scale interconnected systems, these frameworks do not allow to explicitly take the (dynamics of the) environment of such a component into account, possibly making the specification unnecessarily conservative.}

    Motivated by these issues, we introduce \emph{contracts} as specifications for linear dynamical systems. {\color{black} Contracts were initially developed in the field of software engineering \cite{meyer1992} but were later adapted to a variety of system classes in computer science:} rely-guarantee contracts were introduced in \cite{jones1983} to deal with programs that operate concurrently, while the popularisation of interface theories \cite{chakrabarti2002, dealfaro2005} has led to a boom in the development of contract-based theories for cyber-physical systems \cite{davare2013, vincentelli2012}. Consequently, there is a great diversity of styles and approaches to using contracts as specifications, which nevertheless follow a common philosophy, namely, to support the independent design of components within interconnected systems. This philosophy is captured in the meta-theory of contracts introduced in \cite{benveniste2018}. This meta-theory abstracts away the specific design choices made when developing a contract theory, while still formally defining all relevant concepts.

    Inspired by this meta-theory, we define assume-guarantee contracts for linear dynamical systems with inputs and outputs. As the name suggests, these contracts consist of assumptions and guarantees, both of which are linear systems themselves, and which can be used to express specifications for a system in the following sense. First, the assumptions capture the available knowledge on the dynamic behaviour of the environment in which the system is supposed to operate, thus establishing a class of compatible environments. Second, the guarantees capture the required dynamic behaviour of the system when interconnected with a compatible environment, thus establishing a class of implementations. This is formalized very naturally using the behavioural approach to systems theory \cite{willems1989, polderman1998}.

    In addition to defining contracts, we characterize contract implementation and provide a method for verifying that a given linear system implements a given contract. We also define and characterize the concepts of contract refinement and contract conjunction, again taking inspiration from the meta-theory in \cite{benveniste2018}. The notion of contract refinement allows us to reason when one contract represents a stricter specification than another contract. On the other hand, the notion of contract conjunction allows us to construct a contract that fuses the specifications expressed by several different contracts. As will be shown later in this paper, these two concepts are intimately related and can be characterized in a very intuitive and conceptually simple manner.

    We regard this work as a first step towards a comprehensive contract theory for linear dynamical systems. Nevertheless, we note that {\color{black} ideas from contract theories} have already been used to express specifications on dynamical systems. Assume-guarantee contracts that capture set-invariance properties were introduced in \cite{saoud2018} and used in \cite{saoud2018b} for the design of symbolic controllers. Closely related are also the assume-guarantee contracts for safety introduced in \cite{eqtami2019}. However, in contrast to the contracts in this paper, the contracts presented in \cite{saoud2018b} and \cite{eqtami2019} do not allow specifications involving dynamics, neither in the assumptions nor in the guarantees. This is not the case for another class of contracts, called parametric assume-guarantee contracts \cite{kim2017, khatib2020}, which were introduced to express specifications on input-output gain properties. Nonetheless, the latter are only defined for discrete systems, {\color{black} whereas we consider continuous systems} in this paper. In fact, the contracts in this paper are most closely related to the contracts introduced in \cite{besselink2019}, the main difference being that the external variable there is not assumed to be an input-output pair. Our work is also closely related to the work on compositional analysis and assume-guarantee reasoning for linear systems presented in \cite{kerber2009, kerber2010, kerber2011}, although these do not explicitly define contracts.

    The remainder of this paper is structured as follows. In Section~\ref{sec:system_class}, we discuss the class of systems that will be treated in this paper. There, we also review the concept of external behaviour together with some relevant results. In Section~\ref{sec:contracts}, we define assume-guarantee contracts for the class of systems discussed in Section~\ref{sec:system_class}. Moreover, we define and characterize the notions of contract implementation, contract refinement and contract conjunction. As such, Section~\ref{sec:contracts} contains the main contributions of this paper. These are then demonstrated with an illustrative example in Section~\ref{sec:example}, followed by concluding remarks in Section~\ref{sec:conclusion}.

    The notation in this paper is mostly standard. The space of smooth functions from $\bbR$ to $\bbR^n$ is denoted by $\cinf{n}$. A matrix whose entries are polynomials is called a \emph{polynomial matrix}, and a matrix whose entries are rational functions is called a \emph{rational matrix}. All polynomials are univariate and have real coefficients. A rational function is \emph{proper} if the degree of its denominator is greater than or equal to the degree of its numerator. A rational matrix is proper if all of its entries are proper rational functions. We say that a square polynomial matrix $P(s)$ is \emph{invertible} if there exists a rational matrix $Q(s)$ such that $P(s)Q(s) = I$. We say that $P(s)$ is \emph{unimodular} if there exists a \emph{polynomial} matrix $Q(s)$ such that $P(s)Q(s) = I$. In both cases, we say that $Q(s)$ is the \emph{inverse} of $P(s)$, which we denote by $P(s)\inv$.

    \section{Models of physical systems}\label{sec:system_class}

    Consider the linear system
    \begin{equation}\label{eq:sys_iso}
     \sys:
     \left\lbrace
     \begin{aligned}
         \dot x &= Ax + Bu,\\
         y &= Cx + Du,
     \end{aligned}
     \right.
    \end{equation}
    with state trajectory $x\in\cinf{n}$, input trajectory $u\in\cinf{m}$, output trajectory $y\in\cinf{p}$. The system $\sys$ is regarded as an open system in  which  the  external  variables $u$ and $y$ interact  with  the  environment,  whereas the state $x$ is internal and does not interact with the environment.
    \begin{figure}[!htpb]
    	\centering
    	\begin{tikzpicture}[->,>=stealth',shorten >=1pt,auto,node distance=3cm,
    		semithick]
    		\tikzset{box/.style = {shape = rectangle,
    				color=black,
    				fill=white!96!black,
    				text = black,
    				inner sep = 5pt,
    				minimum width = 35pt,
    				minimum height = 17.5pt,
    				draw}
    		}

    		\node[box] (S2) at (2,0) {$\sys$};

    		\draw (0,0) -- node[midway, above] {$u$} (S2);
    		\draw (S2) -- node[midway, above] {$y$} (4,0);
    	\end{tikzpicture}
    	\caption{The system $\sys$ as a signal processor.}
    	\label{fig:sys}
        \vspace{-3.5mm}
    \end{figure}
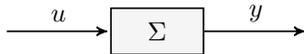
	More precisely, $\sys$ is seen as a signal processor that takes an externally provided input signal $u\in\cinf{m}$ and transforms it to an externally available output signal $y\in\cinf{p}$, as shown by the diagram in Figure~\ref{fig:sys}. Our goal in this paper is to develop a formal method for expressing specifications on the dynamic behaviour of such a system. The approach we take requires a method for system comparison; we need to be able to say when one system behaves ``like'' another system. This is formalized quite naturally using the behavioural approach to systems theory \cite{willems2007a, polderman1998}.  In particular, we define the \emph{external behaviour} $\B{\sys}$ of $\sys$ as the linear subspace
    \begin{equation*}
     	\B{\sys} = \set{(u,y)\in\cinf{m+p}}{\exists x\in\cinf{n} \text{ s.t.\ } \eqref{eq:sys_iso} \text{ holds} }.
    \end{equation*}

    \begin{remark}
        The external behaviour can be used to compare systems in the following sense. Consider two systems $\sys_1$ and $\sys_2$ of the form \eqref{eq:sys_iso}. If $\B{\sys_1}\subset\B{\sys_2}$, then for a given input, the set of outputs generated by $\sys_1$ is contained in the set of outputs generated by $\sys_2$, thus $\sys_2$ can be considered to have ``richer" dynamics than $\sys_1$. Taking this a step further, if $\B{\sys_1} = \B{\sys_2}$, then for a given input, the set of outputs generated by $\sys_1$ is precisely the same as the set of outputs generated by $\sys_2$, hence $\sys_1$ and $\sys_2$ cannot be distinguished on the basis of external behaviour alone. In such a case, we will view $\sys_1$ and $\sys_2$ as different representations of the same external behaviour rather than different systems, i.e., we identify the system with its external behaviour.
    \end{remark}

    There are many different representations of a given external behaviour. As we are interested in providing specifications only on the external behaviour of a system, a representation that does not involve the internal state would be more appropriate for our purposes. With this in mind, consider the linear system
    \begin{equation}\label{eq:sys_io}
    	\sys:
    	\begin{aligned}
	    	P\pddt y = Q\pddt u
    	\end{aligned}
    \end{equation}
    with $u\in\cinf{m}$, $y\in\cinf{p}$ and real polynomial matrices $P(s)$ and $Q(s)$. If $P(s)$ is square and invertible and $P(s)\inv Q(s)$ is a proper rational matrix, then we say that $\sys$ is in \emph{input-output form}  \cite[Section~3.3]{polderman1998}. The system $\sys$ being in input-output form guarantees that $u$ can be chosen freely, i.e., for all $u\in\cinf{m}$, there exists $y\in\cinf{p}$ such that $(u,y)\in\B{\sys}$. In fact, it also guarantees that none of the components of $y$ can be chosen freely, thus ensuring that $u$ and $y$ in \eqref{eq:sys_io} have the roles of input and output, respectively. Similarly to before, we define the external behaviour of $\sys$ of the form \eqref{eq:sys_io} as
    \begin{equation*}
    	\B{\sys} = \set{(u,y)\in\cinf{m+p}}{\eqref{eq:sys_io} \text{ holds} }.
    \end{equation*}
    {\color{black} The external behaviour of a system of the form (1) can be compared to that of a system of the form (2). In fact, we have the following result on representations.}
    \begin{theorem}\cite[Theorem 6.2]{willems1983}\label{thm:representations}
    	Let $\mathfrak{B}\subset\cinf{m+p}$ be a linear subspace. There exists $\sys_1$ of the form \eqref{eq:sys_iso} such that $\B{\sys_1} = \mathfrak{B}$ if and only if there exists $\sys_2$ of the form \eqref{eq:sys_io} in input-output form such that $\B{\sys_2} = \mathfrak{B}$.
    \end{theorem}

    In view of Theorem~1, it makes no difference whether we consider systems of the form~\eqref{eq:sys_iso} or systems of the form~\eqref{eq:sys_io} in input-output form to represent external behaviours. However, the latter are better suited for the type of analysis that will be carried out in the following sections. For example, inclusion of external behaviours has a simple algebraic characterization for systems of the form \eqref{eq:sys_io}, as shown  next.

    \begin{theorem}\label{thm:inclusion}
        Let $\mathfrak{B}_j$, $j\in\{1,2\}$, be defined as
        \begin{equation*}
            \mathfrak{B}_j = \set{w\in\cinf{k}}{R_j\pddt w = 0},
        \end{equation*}
        where $R_j(s)$ is a polynomial matrix.  Then $\mathfrak{B}_1 \subset \mathfrak{B}_2$ if and only if there exists a polynomial matrix $M(s)$ such that $R_2(s) = M(s) R_1(s)$.
    \end{theorem}
    \begin{proof}
        Suppose that there exists a polynomial matrix $M(s)$ satisfying $R_2(s) = M(s) R_1(s)$. Let $w\in\mathfrak{B}_1$. Then
        \begin{equation*}
            R_2\pddt w = M\pddt R_1\pddt w = 0,
        \end{equation*}
        hence $w\in\mathfrak{B}_2$ and thus $\mathfrak{B}_1\subset \mathfrak{B}_2$. Conversely, suppose that $\mathfrak{B}_1\subset\mathfrak{B}_2$, i.e., $R_1\pddt w = 0$ implies $R_2\pddt w = 0$. Using \cite[Lemma~2.1]{polderman2000}, it follows that there exists a polynomial matrix $M(s)$ such that $R_2(s) = M(s) R_1(s)$.
    \end{proof}

    As a consequence of Theorem~\ref{thm:inclusion}, given two systems $\sys_1$ and $\sys_2$ of the form \eqref{eq:sys_io}, we have that $\B{\sys_1}\subset\B{\sys_2}$ if and only if there exists a polynomial matrix $M(s)$ such that
    \begin{equation*}
        \bbm P_2(s) & -Q_2(s) \ebm = M(s)\bbm P_1(s) & -Q_1(s) \ebm.
    \end{equation*}
    The following lemma utilizes the Smith canonical form \cite[Section~1.8]{kaczorek2007} to produce a condition for the existence of such a polynomial matrix $M(s)$.
     \begin{lemma}\label{lem:inclusion}
        Let $R_1(s)$ and $R_2(s)$ be polynomial matrices and assume that $R_1(\lambda)$ has full row rank for some $\lambda\in\bbC$. Moreover, let $U_1(s)$ and $V_1(s)$ be unimodular matrices that bring $R_1(s)$ to its Smith canonical form, i.e., $R_1(s) = U_1(s)\bbm D_1(s) & 0 \ebm V_1(s)$, where $D_1(s)$ is an invertible diagonal polynomial matrix. Then there exists a polynomial matrix $M(s)$ such that $R_2(s) = M(s)R_1(s)$ if and only if the following conditions hold:
        \begin{enumerate}
            \item $R_2(s)V_1(s)\inv\bbm 0 \\ I \ebm = 0$;\\[1mm]
            \item $R_2(s)V_1(s)\inv\bbm D_1(s)\inv \\ 0 \ebm$ is a polynomial matrix.\vspace{1mm}
        \end{enumerate}
    \end{lemma}
    \begin{proof}
        Suppose that there exists a polynomial matrix $M(s)$ that satisfies $R_2(s) = M(s)R_1(s)$. Then we obtain
        \begin{equation*}
            R_2(s)V_1(s)\inv = \bbm M(s)U_1(s)D_1(s) & 0 \ebm,
        \end{equation*}
        hence conditions 1 and 2 hold, the latter because $M(s)U_1(s)$ is a polynomial matrix. For the converse, suppose that conditions 1 and 2 hold. Condition 2 tells us that
        \begin{equation*}
            M(s) = R_2(s)V_1(s)\inv\bbm D_1(s)\inv \\ 0 \ebm U_1(s)\inv
        \end{equation*}
        is a polynomial matrix. In view of condition 1, we obtain
        \begin{align*}
                M(s)R_1(s) &= R_2(s) V_1(s) \bbm I & 0 \\ 0 & 0 \ebm V_1(s)\inv\\
                &= R_2(s) V_1(s) \bbm I & 0 \\ 0 & I \ebm V_1(s)\inv = R_2(s),
        \end{align*}
        which concludes the proof.
    \end{proof}

    Note that if $\sys_1$ of the form \eqref{eq:sys_io} is in input-output form, then $[ P_1(\lambda) \quad Q_1(\lambda) ]$ has full row rank for some $\lambda\in\bbC$ because $P_1(s)$ is invertible, hence the assumption on $R_1(s)$ in Lemma~\ref{lem:inclusion} is satisfied. More generally, given a behaviour
    \begin{equation*}
        \mathfrak{B} = \set{w\in\cinf{k}}{R\pddt w = 0},
    \end{equation*}
    we can always find a polynomial matrix $R'(s)$ such that
    \begin{equation*}
        \mathfrak{B} = \set{w\in\cinf{k}}{R'\pddt w = 0}
    \end{equation*}
    and $R'(\lambda)$ has full row rank for some $\lambda\in\bbC$, hence the assumption on $R_1(s)$ in Lemma~\ref{lem:inclusion} is not restrictive at all.

    We conclude this section with the following remark.
    \begin{remark}\label{rem:latent}
        Although we have chosen to represent external behaviours without involving a state variable, it is often convenient to represent a given external behaviour with the help of so-called \emph{latent variables}. These are variables that are included in the representation of the external behaviour but are not necessarily of interest to us. Therefore, most generally, we will consider systems of the form
        \begin{equation}\label{eq:sys_latent}
            \sys: P\pddt y = Q\pddt u + E\pddt l,
        \end{equation}
        where $l\in\cinf{r}$ is a latent variable and $E(s)$ is a polynomial matrix. Note that $\sys$ of the form \eqref{eq:sys_iso} is actually a latent variable representation of its external behaviour $\B{\sys}$ with the state $x$ as the latent variable. We already know that for all $\sys_1$ of the form \eqref{eq:sys_iso} there exists $\sys_2$ of the form \eqref{eq:sys_io} in input-output form such that $\B{\sys_1} = \B{\sys_2}$. More generally, the latent variable in $\sys_3$ of the form \eqref{eq:sys_latent} can always be eliminated to obtain $\sys_2$ of the form \eqref{eq:sys_io} such that $\B{\sys_3} = \B{\sys_2}$, as follows from \cite[Theorem~6.2.6]{polderman1998}.
    \end{remark}

    \section{Contracts}\label{sec:contracts}
    Consider a system $\sys$ of the form \eqref{eq:sys_io} in input-output form. We want to express specifications on the external behaviour of such a system. To this end, as an open system, we can assume that $\sys$ operates in interconnection with its environment. Having knowledge about this environment can ease the design burden of $\sys$, hence it should be taken into account when expressing specifications. With this in mind, we will assume that the environment of $\sys$ is another system that generates inputs for it. More precisely, an \emph{environment} $\env$ is a system of the form
    \begin{equation}\label{eq:env}
    	\env:  0 = E\pddt u,
    \end{equation}
    where $u\in\cinf{m}$ and $E(s)$ is a real polynomial matrix. The environment $\env$ defines the \emph{input behaviour}
    \begin{equation*}
    	\Bi{\env} = \set{u\in\cinf{m}}{\eqref{eq:env} \text{ holds}}.
    \end{equation*}
	The interconnection of $\sys$ with the environment $\env$ is obtained by setting the input generated by $\env$ as input of $\sys$. This results in the interconnection
    \begin{equation}\label{eq:env_sys}
    	\env\meet\sys: \bbm P\pddt \\[1mm] 0 \ebm y = \bbm Q\pddt \\[1mm] E\pddt \ebm u,
    \end{equation}
    which is represented graphically in Figure~\ref{fig:interconnection}.
    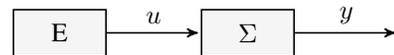
\begin{figure}[!htpb]
        \centering
        \begin{tikzpicture}[->,>=stealth',shorten >=1pt,auto,node distance=3cm,
            semithick]
            \tikzset{box/.style = {shape = rectangle,
                    color=black,
                    fill=white!96!black,
                    text = black,
                    inner sep = 5pt,
                    minimum width = 35pt,
                    minimum height = 17.5pt,
                    draw}
            }

%            \node at (-1,0) {};
            \node[box] (S1) at (0,0) {$\env$};
            \node[box] (S2) at (2.5,0) {$\sys$};

            \draw (S1) -- node[midway, above] {$u$} (S2);
            \draw (S2) -- node[midway, above] {$y$} (4.5,0);

%            \draw[-,dashed] (0,0.75) -- (4,0.75) -- (4,-1) -- (0,-1) -- (0,0.75);
%            \node at (2,-0.75) {$\env\meet\sys$};
        \end{tikzpicture}
        \caption{The interconnection $\env\meet\sys$.}
        \label{fig:interconnection}
        \vspace{-3mm}
    \end{figure}

	As a means of expressing specifications on the external behaviour of $\sys$, we are  interested  in  guaranteeing  properties of $\sys$ when interconnected  with  relevant  environments $\env$. In particular, if we define the \emph{output behaviour}
    \begin{equation*}
    	\Bo{\env\meet\sys} = \set{y\in\cinf{p}}{\eqref{eq:env_sys} \text{ holds}},
    \end{equation*}
    then we want to guarantee properties of $\Bo{\env\meet\sys}$ for all relevant environments $\env$. We will make this explicit by introducing two systems: \emph{assumptions} $\ass$ and \emph{guarantees} $\gar$. Assumptions $\ass$ are a system of the form
    \begin{equation}\label{eq:ass}
    	\ass: 0 = A\pddt u,
    \end{equation}
    and, like environments, they represent the input behaviour
    \begin{equation*}
        \Bi{\ass} = \set{u\in\cinf{m}}{\eqref{eq:ass} \text{ holds}}.
    \end{equation*}
    On the other hand, guarantees $\gar$ are a system of the form
    \begin{equation}\label{eq:gar}
    	\gar: G\pddt y = 0,
    \end{equation}
    and they represent the output behaviour
    \begin{equation*}
    	\Bo{\gar} = \set{y\in\cinf{p}}{\eqref{eq:gar} \text{ holds}}.
    \end{equation*}
    \begin{remark}
    	Although we define environments, assumptions and guarantees as systems involving a single variable, we might represent the behaviours they define with the help of latent variables, as in the general form \eqref{eq:sys_latent}. {\color{black} Since latent variables can always be eliminated (recall Remark~\ref{rem:latent}), the representations \eqref{eq:env}, \eqref{eq:ass} and \eqref{eq:gar} do not pose a restriction.}
    \end{remark}

    With assumptions and guarantees defined, we are ready to introduce the definition of a contract.
    \begin{definition}\label{def:contract}
    	A \emph{contract} $\con$ is a pair $(\ass, \gar)$ of assumptions $\ass$ and guarantees $\gar$.
    \end{definition}

    Contracts can be used to express formal specifications for the external behaviour of systems $\sys$ of the form \eqref{eq:sys_io} in input-output form, as captured in the following definition.

    \begin{definition}
    	An environment $\env$ is \emph{compatible} with the contract $\contract{}$ if $\Bi{\env}\subset\Bi{\ass}$. A system $\sys$ of the form \eqref{eq:sys_io} in input-output form \emph{implements} $\con$ if $\Bo{\env\meet\sys}\subset\Bo{\gar}$ for all environments $\env$ compatible with $\con$. In such a case, we say that $\sys$ is an \emph{implementation} of $\con$.
    \end{definition}

	In other words, an environment is compatible with a contract if the inputs that it generates can be generated by the assumptions, and a system implements a contract if the outputs that it generates when interconnected with any compatible environment can be generated by the guarantees. A contract thus  gives  a  formal  specification  for  the external behaviour of a system through two aspects. First, it specifies (by the assumptions) the class of environments in which the system is supposed to operate. Second, it characterizes  the  required external behaviour  of the system through  the  guarantees,  which  the  system  needs  to  satisfy for  any compatible environment. {\color{black} We emphasize that both the assumptions and guarantees are dynamical systems, hence the specification that a contract expresses is also dynamic.}

	Although contract implementation is defined using the class of compatible environments, verifying whether a system $\sys$ of the form \eqref{eq:sys_io} in input-output form is an implementation can be done directly via the assumptions and guarantees, i.e., without explicitly constructing the class of compatible environments. This is stated in the following theorem, which is represented graphically in Figure~\ref{fig:implementation}.
	\begin{theorem}\label{thm:implementation}
		A system $\sys$ of the form \eqref{eq:sys_io} is an implementation of $\contract{}$ if and only if $\Bo{\ass\meet\sys} \subset \Bo{\gar}$.
	\end{theorem}
	\begin{proof}
		Suppose that $\Bo{\ass\meet\sys}\subset\Bo{\gar}$. Note that $y\in\Bo{\ass\meet\sys}$ if and only if there exists $u\in\Bi{\ass}$ such that $(u,y)\in\B{\sys}$. With this in mind, let  $\env$ be an environment compatible with $\con$, and let $y\in\Bo{\env\meet\sys}$. Then there exists $u\in\Bi{\env}\subset\Bi{\ass}$ such that $(u,y)\in\B{\sys}$, hence $y\in\Bo{\ass\meet\sys}\subset\Bo{\gar}$. This shows that $\Bo{\env\meet\sys}\subset\Bo{\gar}$ for all compatible environments $\env$ and thus $\sys$ is an implementation of $\con$.	Conversely, suppose that $\sys$ is an implementation of $\calC$. Since $\ass$ is a compatible environment of $\con$, it follows that $\Bo{\ass\meet\sys}\subset\Bo{\gar}$.
	\end{proof}
    \begin{figure}
        \centering
        \begin{tikzpicture}[->,>=stealth',shorten >=1pt,auto,node distance=3cm,
            semithick]
            \tikzset{box/.style = {shape = rectangle,
                    color=black,
                    fill=white!96!black,
                    text = black,
                    inner sep = 5pt,
                    minimum width = 35pt,
                    minimum height = 17.5pt,
                    draw}
            }
            \draw[-,fill=white!96!black] (0,0.75) -- (4,0.75) -- (4,-.825) -- (0,-.825) -- (0,0.75);

            \node[box, fill=white] (S1) at (1,0) {$\ass$};
            \node[box, fill=white] (S2) at (3,0) {$\sys$};
            \node at (2,-0.625) {$\ass\meet\sys$};

            \draw[->,>=stealth',shorten >=1pt,auto,node distance=3cm,
            semithick, dashed] (S1) -- node[midway, above] {$u$} (S2);
            \draw[-,dashed] (S2) -- (4,0);
            \draw (4,0) -- node[midway, above] {$y$} (4.75,0);

            \node at (5.125,0) {$\subset$};
            \node[box] (G) at (6.25,0) {$\gar$};
            \draw (G) -- node[midway, above] {$y$} (7.625,0);
        \end{tikzpicture}
        \caption{Implementation of $\contract{}$.}
        \label{fig:implementation}
        \vspace{-7mm}
    \end{figure}

    \begin{remark}\label{rem:implementation}
        Theorem~\ref{thm:implementation} allows us to verify that $\sys$ of the form \eqref{eq:sys_io} is an implementation of $\contract{}$ provided that we can verify that $\Bo{\ass\meet\sys}\subset\Bo{\gar}$. The latter can be done in two steps. First, in view of \cite[Theorem~6.2.6]{polderman1998}, we can eliminate the variable $u$ from $\ass\meet\sys$ to obtain
        \begin{equation}\label{eq:ass_meet_sys_R}
            \Bo{\ass\meet\sys} = \set{y\in\cinf{p}}{R\pddt y = 0}
        \end{equation}
        for some polynomial matrix $R(s)$, where $R(\lambda)$ has full row rank for some $\lambda\in \bbC$. Second, from Theorem~\ref{thm:inclusion} we know that  $\Bo{\ass\meet\env}\subset\Bo{\gar}$ if and only if there exists a polynomial matrix $M(s)$ such that $G(s) = M(s)R(s)$, where the latter can be verified using Lemma~\ref{lem:inclusion}. Note that due to Theorem~\ref{thm:representations}, Theorem~\ref{thm:implementation} is also valid for systems $\sys$ of the form \eqref{eq:sys_iso}. For the latter, we need to eliminate $x$ as well as $u$ in order to obtain a polynomial matrix $R(s)$ such that \eqref{eq:ass_meet_sys_R} holds.
    \end{remark}

    A distinguishing feature of using contracts for specifications is that contracts themselves can be compared through a notion of refinement.
    \begin{definition}
    	The contract $\con_1$ \emph{refines} the contract $\con_2$, denoted by $\con_1\simby\con_2$, if all compatible environments of $\con_2$ are compatible environments of $\con_1$ and all implementations of $\con_1$ are implementations of $\con_2$.
    \end{definition}

    We have that $\con_1$ refines $\con_2$ if it specifies stricter guarantees that have to be satisfied for a larger class of environments. In such a case, it is clear that $\con_1$ expresses a more restrictive specification than $\con_2$. Just like contract implementation, contract refinement can be verified on the basis of assumptions and guarantees alone, i.e., without explicitly constructing the classes of compatible environments and implementations of the two contracts. This is the content of the following theorem, whose proof can be found in the appendix.
    \begin{theorem}\label{thm:refinement}
    	The contract $\contract{1}$ refines the contract $\contract{2}$ if and only if $\Bi{\ass_2}\subset\Bi{\ass_1}$ and $\Bo{\gar_1}\subset\Bo{\gar_2}$.
    \end{theorem}

    If a contract $\con$ refines both contracts $\con_1$ and $\con_2$, then $\con$ captures both the specifications that $\con_1$ expresses and the specifications that $\con_2$ expresses. This suggests that multiple contracts can be combined into a single contract that represents a fusion of their specifications. This motivates the following definition.
   	\begin{definition}
   		The \emph{conjunction} of contracts $\con_1$ and $\con_2$, denoted by $\con_1\meet\con_2$, is the largest (with respect to contract refinement) contract that refines both $\con_1$ and $\con_2$.
   	\end{definition}

   	Note that the largest contract that refines both $\con_1$ and $\con_2$ corresponds to the least restrictive contract that fuses the specifications of $\con_1$ and $\con_2$. This suggests that the guarantees of $\con_1\meet\con_2$ should be the guarantees common to $\con_1$ and $\con_2$ and nothing more, and the assumptions of $\con_1\meet\con_2$ should include the assumptions of both $\con_1$ and $\con_2$ and nothing less. The following lemma captures some of this intuition.

	\begin{lemma}\label{lem:conjunction}
		A contract $\contract{}$ refines both contracts $\contract{1}$ and $\contract{2}$ if and only if $\Bi{\ass_1}+\Bi{\ass_2}\subset \Bi{\ass}$ and $\Bo{\gar} \subset \Bo{\gar_1}\cap\Bo{\gar_2}$.
	\end{lemma}
	\begin{proof}
		Due to Theorem~\ref{thm:refinement}, the contract $\con$ refines both contracts $\con_1$ and $\con_2$ if and only if $\Bi{\ass_i}\subset\Bi{\ass}$ and $\Bo{\gar}\subset\Bo{\gar_i}$ for all $i\in\{1,2\}$. Since behaviours are linear, this is the case if and only if $\Bi{\ass_1}+\Bi{\ass_2}\subset \Bi{\ass}$ and $\Bo{\gar} \subset \Bo{\gar_1}\cap\Bo{\gar_2}$.
	\end{proof}

    In view of Theorem~\ref{thm:refinement}, Lemma~\ref{lem:conjunction} tells us that if there exist assumptions $\ass$ with $\Bi{\ass} = \Bi{\ass_1}+\Bi{\ass_2}$ and guarantees $\gar$ with $\Bo{\gar} = \Bo{\gar_1}\cap\Bo{\gar_2}$, then $(\ass,\gar)$ is the largest contract that refines both $\contract{1}$ and $\contract{2}$, hence $\con_1\meet\con_2 = (\ass,\gar)$. Fortunately, such assumptions and guarantees exist and are defined below.
    \begin{definition}\label{def:join_meet}
    	The \emph{join} of assumptions $\ass_1$ and $\ass_2$, denoted by $\ass_1\join \ass_2$, is defined as the assumptions
    	\begin{equation*}
    		\ass_1\join\ass_2: \bbm I & I \\ A_1\pddt & 0 \\ 0 & A_2\pddt \ebm \bbm l_1 \\ l_2 \ebm = \bbm I \\ 0 \\ 0 \ebm u.
    	\end{equation*}
    	The \emph{meet} of guarantees $\gar_1$ and $\gar_2$, denoted by $\gar_1 \meet \gar_2$, is defined as the guarantees
    	\begin{equation*}
    		\gar_1\meet\gar_2: \bbm G_1 \pddt \\[1mm] G_2\pddt \ebm y = 0.
    		\vspace{2mm}
    	\end{equation*}
    \end{definition}

    We remark that although the join $\ass_1\join \ass_2$ is represented with latent variables, these can be eliminated to obtain assumptions as in \eqref{eq:ass}. However, the latent variable representation clearly indicates that every $u\in\Bi{\ass_1\join\ass_2}$ is obtained by summing $l_1\in\Bi{\ass_1}$ and $l_2\in\Bi{\ass_2}$. Having defined the join and meet, we can now write the conjunction of contracts $\con_1$ and $\con_2$ explicitly.
    \begin{theorem}\label{thm:conjunction}
    	The conjunction of contracts $\contract{1}$ and $\contract{2}$ is given by
    	\begin{equation*}
    		\con_1\meet\con_2 = (\ass_1\join\ass_2,\gar_1\meet\gar_2).
    		\vspace{2mm}
    	\end{equation*}
    \end{theorem}
	\begin{proof}
		It is easy to see that $\Bi{\ass_1\join\ass_2} = \Bi{\ass_1}+\Bi{\ass_2}$ and $\Bo{\gar_1\meet\gar_2} = \Bo{\gar_1}\cap\Bo{\gar_2}$. Therefore, due to Lemma~\ref{lem:conjunction} and Theorem~\ref{thm:refinement}, $(\ass_1\join\ass_2,\gar_1\meet\gar_2)$ is the largest contract that refines both $\con_1$ and $\con_2$.
	\end{proof}

    Theorem~\ref{thm:conjunction} allows us to check when $\sys$ of the form \eqref{eq:sys_io} in input-output form is simultaneously an implementation of $\con_1$ and $\con_2$. Indeed, this is the case if and only if $\sys$ implements the conjunction $\con_1\meet\con_2$, whose assumptions and guarantees we can compute explicitly. Given the assumptions and guarantees of $\con_1\meet\con_2$, we have already explained how to verify that $\sys$ implements $\con_1\meet\con_2$ in Remark~\ref{rem:implementation}.

    \section{Illustrative example}\label{sec:example}

    In this section, we will illustrate the concepts and results from last section with a simple example. To this end, suppose that we need to design a car component that does not vibrate while the car is moving, e.g., the driver's seat. For simplicity, we only consider the vertical motion of the car, which we describe by a quarter car model, as shown in Figure~\ref{fig:quarter_car}. The model consists of two masses: the mass $m_1$ of the body and the mass $m_2$ of the wheel. The wheel is attached to the body of the car through a spring with constant $k_1$ and a damper with coefficient $b$. The wheel is connected to the ground by a tire represented as a spring with constant $k_2$. Let $u_1$ and $u_2$ be the vertical positions of the body and the wheel, respectively, and let $l$ be the reference signal from the ground. Then the dynamics of $u_1$ and $u_2$ are given by
    \begin{align}
        \vspace{-1mm}
        m_1 \ddot u_1 &= -b(\dot u_1 - \dot u_2) - k_1(u_1 - u_2),  \label{eq:m1}\\
        m_2 \ddot u_2 &= -b(\dot u_2 - \dot u_1) - k_1(u_2 - u_1) - k_2(u_2 - l). \label{eq:m2}
        \vspace{-1mm}
    \end{align}

    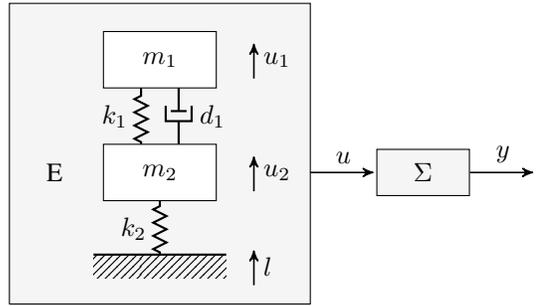
\begin{figure}
        \centering
        \begin{tikzpicture}
            \tikzstyle{spring}=[thick,decorate,decoration={zigzag,pre length=0.3mm,post length=0.3mm,segment length=5}]
            \tikzstyle{damper}=[thick,decoration={markings,
                mark connection node=dmp,
                mark=at position 0.5 with
                {
                    \node (dmp) [thick,inner sep=0pt,transform shape,rotate=-90,minimum width=9pt,minimum height=3pt,draw=none] {};
                    \draw [thick] ($(dmp.north east)+(2pt,0)$) -- (dmp.south east) -- (dmp.south west) -- ($(dmp.north west)+(2pt,0)$);
                    \draw [thick] ($(dmp.north)+(0,-3pt)$) -- ($(dmp.north)+(0,3pt)$);
                }
            }, decorate]

            \tikzstyle{ground}=[fill,pattern=north east lines,draw=none, minimum width=0.75cm, minimum height=0.3cm]

            \tikzset{box/.style = {shape = rectangle,
                    color=black,
                    fill=white!96!black,
                    text = black,
                    inner sep = 5pt,
                    minimum width = 35pt,
                    minimum height = 17.5pt,
                    draw}
            }

            \draw[fill=white!96!black] (-2, -.25) -- (2, -.25) -- (2,3.75) -- (-2, 3.75) -- (-2, -.25);

            \node[ground, minimum width=1.75cm] (g) at (0,.25) {};
            \draw[thick] (g.north east) -- (g.north west);

            \node[minimum height=.75cm, minimum width=1.5cm, draw, fill=white] (m2) at (0,1.5) {$m_2$};

            \node[minimum height=.75cm, minimum width=1.5cm, draw, fill=white] (m1) at (0,3) {$m_1$};

            \draw [spring] (g) -- (m2);
            \draw [spring] (-.25,1.875) -- (-.25,2.625);
            \draw [damper] (.25,1.875) -- (.25,2.625);
            \node at (.70, 2.25) {$d_1$};
            \node at (-.60, 2.25) {$k_1$};
            \node at (-.35, 0.75) {$k_2$};

            \draw[->,>=stealth',shorten >=1pt, semithick] (1.25, 2.75) --node[midway, right] {$u_1$} (1.25,3.25);
            \draw[->,>=stealth',shorten >=1pt, semithick] (1.25, 1.25) --node[midway, right] {$u_2$} (1.25,1.75);
            \draw[->,>=stealth',shorten >=1pt, semithick] (1.25, 0) --node[midway, right] {$l$} (1.25,.5);

            \node at (-1.4, 1.5) {$\env$};
            \node[box] (S) at (3.5, 1.5) {$\sys$};
            \draw[->,>=stealth',shorten >=1pt, semithick] (2, 1.5) --node[midway, above] {$u$} (S);
            \draw[->,>=stealth',shorten >=1pt, semithick] (S) --node[midway, above] {$y$} (5, 1.5);
        \end{tikzpicture}
        \caption{Quarter car model of the environment.}
        \label{fig:quarter_car}
        \vspace{-5mm}
    \end{figure}

    Now, suppose that the vertical position of our component is controlled electronically on the basis of the positions of the body and the wheel. In other words, the input to our system is given by $u = [u_1 \quad u_2]\t$. The component, with vertical position $y$, will be devoid of vibrations if it has zero acceleration, hence the desired guarantees $\gar$ are given by \eqref{eq:gar} with $G(s) = s^2$. If there is no information about the reference signal $l$, i.e., there is no information about the surface on which the car moves, then the available information about $u$ is given by~\eqref{eq:m1}. Therefore, the assumptions $\ass$ are given by $\eqref{eq:ass}$ with $A(s) = \bbm m_1s^2 + bs + k_1 & -bs - k_1 \ebm$ and our specification is captured by the contract $\contract{}$.

    An implementation of $\con$ is given by the system
    \begin{equation*}
        \sys: \left\lbrace
        \begin{aligned}
            \dot x &= \bbm 0 & 1 \\ 0 & 0 \ebm x + \tfrac{1}{m_1}\bbm b & -b \\ k_1 & -k_1\ebm u,\\[2mm]
            y &= \bbm 1 & 0 \ebm x + \bbm 1 & 0 \ebm u\bstrut.
        \end{aligned}\right.
    \end{equation*}
    Indeed, we can eliminate $x$ from $\sys$ to obtain
    \begin{equation*}
        \ddt{2} y  = \ddt{2}u_1 + \tfrac{b}{m_1}\ddt{}(u_1 - u_2) + \tfrac{k_1}{m_1}(u_1 - u_2)
    \end{equation*}
    where the right-hand side is zero for all $u\in\Bi{\ass}$. Then
    \begin{equation*}
        \Bo{\ass\meet\sys} = \set{y\in\cinf{p}}{\ddt{2} y = 0},
    \end{equation*}
    and thus $\Bo{\ass\meet\sys} = \Bo{\gar}$, which shows that $\sys$ is an implementation of $\con$ because of Theorem~\ref{thm:implementation}.

    Alternatively, if there is information about the reference signal $l$, then we can use that information to construct a contract which is easier to implement. For example, if we assume that $l = 0$, then \eqref{eq:m1} and \eqref{eq:m2} yield the assumptions
    \begin{equation*}
        \ass_0: A_0\pddt u = 0,
    \end{equation*}
    where the polynomial matrix $A_0(s)$ is given by
    \begin{equation*}
        A_0(s) = \bbm m_1s^2 + bs + k_1 & -bs - k_1 \\  -bs - k_1 & m_2s^2 + bs + k_1 + k_2\ebm.
    \end{equation*}
    Consequently, the contract $\con_0 = (\ass_0 ,\gar)$ is implemented by
    \begin{equation*}
        \sys_0: \left\lbrace
        \begin{aligned}
            \dot x &= \bbm 0 & 1 \\ 0 & 0 \ebm x + \tfrac{1}{m_2}\bbm -b & b \\ -k_1 & k_1 + k_2\ebm u,\\[2mm]
            y &= \bbm 1 & 0 \ebm x + \bbm 0 & 1 \ebm u\bstrut.
        \end{aligned}\right.
    \end{equation*}
    Indeed, eliminating $x$ from $\sys_0$ yields
    \begin{equation*}
        \ddt{2} y  = \ddt{2}u_2 + \tfrac{b}{m_2}\ddt{}(u_2 - u_1) + \tfrac{k_1}{m_2}(u_2 - u_1) + \tfrac{k_2}{m_2}u_2,
    \end{equation*}
    where the right-hand side is zero for all $u\in\Bi{\ass_0}$, hence $\Bo{\ass_0\meet\sys_0} = \Bo{\gar}$. Note that $\sys_0$ does not implement $\con$ because the right-hand side of the latter is not necessarily zero for all $u\in\Bi{\ass}$. However, we clearly have that $\Bi{\ass_0}\subset\Bi{\ass}$, hence $\con$ refines $\con_0$ due to Theorem~\ref{thm:refinement}. This implies that $\sys$ implements $\con_0$, and more generally, that any implementation of $\con$ is an implementation of $\con_0$. Of course, we already expected a component that works for arbitrary reference signal $l$ to work for $l=0$ in particular.

    Finally, we might want to take into account different possibilities for the reference signal $l$. For example, we might want to consider reference signals that satisfy $\dot l = 0$ (``flat'' road) and $\ddot l + l =0$ (``wavy'' road). Substituting these in equations \eqref{eq:m1} and \eqref{eq:m2} results in the assumptions
    \begin{equation*}
        \ass_1: A_1\pddt u = 0,\quad \ass_2: A_2\pddt u = 0,
    \end{equation*}
    with polynomial matrices $A_1(s)$ and $A_2(s)$ given by
    \begin{equation*}
        A_1(s) = \bbm 1 & 0 \\ 0 & s\ebm A_0(s),\quad A_2(s) = \bbm 1 & 0 \\ 0 & s^2 + 1 \ebm A_0(s).
    \end{equation*}
    Then a component that works for a ``flat'' road needs to satisfy the contract $\con_1 = (\ass_1, \gar)$, and a component that works for a ``wavy'' road needs to satisfy the contract $\con_2 = (\ass_2, \gar)$. Therefore, a component that works for both roads needs to satisfy the conjunction $\con_1\meet\con_2 = (\ass_1\join\ass_2, \gar)$, which we have obtained using Theorem~\ref{thm:conjunction}. Note that $\con$ refines both $\con_1$ and $\con_2$, while $\con_1$ and $\con_2$ both refine $\con_0$. Therefore, $\con$ refines the conjunction $\con_1\meet\con_2$, which in turn refines $\con_0$.

    \section{Conclusion}\label{sec:conclusion}
    We have introduced assume-guarantee contracts for linear dynamical systems with inputs and outputs. We defined these as a pair of linear systems, called assumptions and guarantees, which were used to characterize the class of compatible environments and the class of implementations through the notion of system behaviour, and in particular, by inclusion of behaviours. In addition to defining contracts, we also characterized contract implementation and proposed a method for verifying it. Moreover, we characterized contract refinement by mirrored inclusions of behaviours of assumptions and guarantees. Using this, we provided an explicit characterization of contract conjunction in terms of the join of assumptions and meet of guarantees. We also demonstrated our setup and results with an illustrative example. Finally, we still require a suitable notion of contract composition. This would enable the component-based analysis and design of interconnected systems, hence its definition and characterization will be the focus of future work.

	\appendix
    Before we give the proof of Theorem~\ref{thm:refinement}, we will state and prove an intermediate result involving autonomous systems. A system $\sys$ is called \emph{autonomous} if
	\begin{equation*}
		\Bo{\sys} = \set{y\in\cinf{p}}{P\pddt y = 0}
	\end{equation*}
	for some square and invertible polynomial matrix $P(s)$.
	\begin{lemma}\label{lem:refinement}
		If $\Bo{\sys}\subset\Bo{\gar_1}$ implies $\Bo{\sys}\subset\Bo{\gar_2}$ for all autonomous $\sys$, then $\Bo{\gar_1}\subset\Bo{\gar_2}$.
	\end{lemma}
	\begin{proof}
		Let $\gar_1$ and $\gar_2$ be given by
		\begin{equation*}
		\gar_1: G_1\pddt y = 0 \qand \gar_2: G_2\pddt y = 0
		\end{equation*}
		for some polynomial matrices $G_1(s)$ and $G_2(s)$. Furthermore, let $U(s)$ and $V(s)$ be unimodular matrices that bring $G_1(s)$ to its Smith canonical form, that is,
		\begin{equation*}
			G_1(s) = U(s)\bbm G_{11}(s) & 0 \\ 0 & 0\ebm V(s)
		\end{equation*}
		where $G_{11}(s)$ is an invertible diagonal polynomial matrix. Then, for any positive integer $k$, the system
        \begin{equation*}
            \sys_k: \bbm G_{11}\pddt & 0\\ 0 & \frac{\d^k}{\d t^k} I \ebm V\pddt y = 0
        \end{equation*}
        is autonomous and, due to Theorem~\ref{thm:inclusion} and the fact that
        \begin{equation}\label{eq:G1}
            G_1(s) = U(s) \bbm I & 0 \\ 0 & 0 \ebm \bbm G_{11}(s) & 0\\ 0 & s^k I \ebm V(s),
        \end{equation}
        it is such that $\Bo{\sys_k}\subset\Bo{\gar_1}$. From this it follows that $\Bo{\sys_k}\subset\Bo{\gar_2}$, hence, for any positive integer $k$, there exists a polynomial matrix $M_k(s)$ such that
        \begin{equation}\label{eq:G2}
            G_2(s) = M_k(s)\bbm G_{11}(s) & 0\\ 0 & s^k I \ebm V(s).
        \end{equation}
        We can partition $M_k(s) = [M_{k1}(s)\quad M_{k2}(s)]$ to obtain
        \begin{equation*}
            G_2(s)V(s)\inv = \bbm M_{k1}(s)G_{11}(s) & M_{k2}(s)s^k \ebm.
        \end{equation*}
        {\color{black}Since $G_2(s)V(s)\inv$ is independent of $k$, it follows that $M_{k2}(s)s^k$ is also independent of $k$. In particular, the degree of $M_{k2}(s)s^k$, defined as the maximum of the degrees of its entries, is the same for all $k$. But then taking $k$ to be greater than this degree implies that $M_{k2}(s) = 0$. Indeed, if $M_{k2}(s) \neq 0$, then the degree of $M_{k2}(s)s^k$ is greater than or equal to $k$, hence it is greater than itself, which is a contradiction. This implies that $M_{k2}(s)s^k = 0$ for some $k$, which implies that $M_{k2}(s)s^k = 0$ for all $k$ since $M_{k2}(s)s^k$ is independent of $k$. Let $k$ be fixed. Then $M_{k2}(s) = 0$ because $M_{k2}(s)s^k = 0$ and we can write }
        \begin{equation*}
            M_k(s) = M_k(s)\bbm I & 0 \\ 0 & 0 \ebm = M_k(s)U(s)\inv U(s) \bbm I & 0 \\ 0 & 0 \ebm.
        \end{equation*}
        Finally, it follows that \eqref{eq:G2} can be rewritten as
        \begin{equation*}
            G_2(s) = M_k(s)U(s)\inv U(s) \bbm I & 0 \\ 0 & 0 \ebm\bbm G_{11}(s) & 0\\ 0 & s^k I \ebm V(s),
        \end{equation*}
        which yields $G_2(s) = M_k(s)U(s)\inv G_1(s)$ because of \eqref{eq:G1}, hence  $\Bo{\gar_1}\subset\Bo{\gar_2}$ due to Theorem~\ref{thm:implementation}.
	\end{proof}

    Now we can turn to the proof of Theorem~\ref{thm:refinement}.

    \begin{proof}
        Suppose that $\Bi{\ass_2}\subset\Bi{\ass_1}$ and $\Bo{\gar_1}\subset\Bo{\gar_2}$. Let $\env$ be an environment compatible with $\con_2$. Then $\Bi{\env}\subset\Bi{\ass_2}\subset\Bi{\ass_1}$, hence $\env$ is compatible with $\con_2$. On the other hand, let $\sys$ be an implementation of $\con_1$. Note that $\ass_2$ is an environment compatible with $\con_1$, hence $\Bo{\ass_2\meet\sys}\subset\Bo{\gar_1}\subset\Bo{\gar_2}$ and thus $\sys$ is an implementation of $\con_2$.

        Conversely, suppose that $\con_1$ refines $\con_2$. Since $\ass_2$ is an environment compatible with $\con_2$, it follows that $\ass_2$ is compatible with $\con_1$ and $\Bi{\ass_2}\subset\Bi{\ass_1}$. Note that if $\sys$ is autonomous, then
        \begin{equation*}
            \Bo{\ass_1\meet\sys} = \Bo{\ass_2\meet\sys} = \Bo{\sys}.
        \end{equation*}
        In view of Theorem~\ref{thm:implementation}, this means that an autonomous $\sys$ is an implementation of $\con_i$, $i\in\{1,2\}$, if and only if $\Bo{\sys}\subset\Bo{\gar_i}$. Since all implementations of $\con_1$ are implementations of $\con_2$, it follows that $\Bo{\sys}\subset\Bo{\gar_1}$ implies $\Bo{\sys}\subset\Bo{\gar_2}$ for all autonomous $\sys$ and thus $\Bo{\gar_1}\subset\Bo{\gar_2}$ because of Lemma~\ref{lem:refinement}.
    \end{proof}

	\bibliographystyle{ieeetr}
	\bibliography{../../../references/all}

\begin{thebibliography}{10}

\bibitem{willems1972}
J.~C. Willems, ``{Dissipative dynamical systems part I: General theory},'' {\em
  Archive for Rational Mechanics and Analysis}, vol.~45, no.~5, pp.~321--351,
  1972.

\bibitem{blanchini1999}
F.~Blanchini, ``Set invariance in control,'' {\em Automatica}, vol.~35, no.~11,
  pp.~1747--1767, 1999.

\bibitem{meyer1992}
B.~{Meyer}, ``Applying `design by contract','' {\em Computer}, vol.~25, no.~10,
  pp.~40--51, 1992.

\bibitem{jones1983}
C.~Jones, ``Specification and design of (parallel) programs.,'' in {\em
  Proceedings Of {IFIP} Congress '}, vol.~83, pp.~321--332, 1983.

\bibitem{chakrabarti2002}
A.~Chakrabarti, L.~de~Alfaro, T.~A. Henzinger, and F.~Y.~C. Mang, ``Synchronous
  and bidirectional component interfaces,'' in {\em Computer Aided
  Verification} (E.~Brinksma and K.~G. Larsen, eds.), pp.~414--427, Springer
  Berlin Heidelberg, 2002.

\bibitem{dealfaro2005}
L.~de~Alfaro and T.~A. Henzinger, ``Interface-based design,'' in {\em
  Engineering Theories of Software Intensive Systems} (M.~Broy,
  J.~Gr{\"u}nbauer, D.~Harel, and T.~Hoare, eds.), pp.~83--104, Springer
  Netherlands, 2005.

\bibitem{davare2013}
A.~Davare, D.~Densmore, L.~Guo, R.~Passerone, A.~L. Sangiovanni-Vincentelli,
  A.~Simalatsar, and Q.~Zhu, ``{metroII}: A design environment for
  cyber-physical systems,'' {\em ACM Transactions on Embedded Computing
  Systems}, vol.~12, no.~1s, 2013.

\bibitem{vincentelli2012}
A.~Sangiovanni-Vincentelli, W.~Damm, and R.~Passerone, ``Taming dr.
  {Frankenstein}: {Contract-based design for cyber-physical systems},'' {\em
  European Journal of Control}, vol.~18, no.~3, pp.~217--238, 2012.

\bibitem{benveniste2018}
A.~Benveniste, B.~Caillaud, D.~Nickovic, R.~Passerone, J.-B. Raclet,
  P.~Reinkemeier, A.~Sangiovanni-Vincentelli, W.~Damm, T.~A. Henzinger, and
  K.~G. Larsen, {\em Contracts for System Design}.
\newblock Foundations and Trends in Electronic Design Automation, Now
  Publishers, 2018.

\bibitem{willems1989}
J.~C. Willems, ``Models for dynamics,'' in {\em Dynamics Reported: A Series in
  Dynamical Systems and Their Applications} (U.~Kirchgraber and H.~O. Walther,
  eds.), pp.~171--269, Vieweg+Teubner Verlag, 1989.

\bibitem{polderman1998}
J.~W. Polderman and J.~C. Willems, {\em Introduction to Mathematical Systems
  Theory}.
\newblock Springer-Verlag New York, 1998.

\bibitem{saoud2018}
A.~{Saoud}, A.~{Girard}, and L.~{Fribourg}, ``On the composition of discrete
  and continuous-time assume-guarantee contracts for invariance,'' in {\em
  Proceedings of the European Control Conference}, pp.~435--440, 2018.

\bibitem{saoud2018b}
A.~{Saoud}, A.~{Girard}, and L.~{Fribourg}, ``Contract based design of symbolic
  controllers for interconnected multiperiodic sampled-data systems,'' in {\em
  Proceedings of the IEEE Conference on Decision and Control}, pp.~773--779,
  2018.

\bibitem{eqtami2019}
A.~{Eqtami} and A.~{Girard}, ``A quantitative approach on assume-guarantee
  contracts for safety of interconnected systems,'' in {\em Proceedings of the
  European Control Conference}, pp.~536--541, 2019.

\bibitem{kim2017}
E.~S. Kim, M.~Arcak, and S.~A. Seshia, ``A small gain theorem for parametric
  assume-guarantee contracts,'' in {\em Proceedings of the 20th International
  Conference on Hybrid Systems: Computation and Control}, HSCC '17,
  pp.~207--216, 2017.

\bibitem{khatib2020}
M.~{Al Khatib} and M.~{Zamani}, ``Controller synthesis for interconnected
  systems using parametric assume-guarantee contracts,'' in {\em Proceedings of
  the American Control Conference}, pp.~5419--5424, 2020.

\bibitem{besselink2019}
B.~Besselink, K.~H. Johansson, and A.~van~der Schaft, ``Contracts as
  specifications for dynamical systems in driving variable form,'' in {\em
  Proceedings of the European Control Conference}, pp.~263--268, 2019.

\bibitem{kerber2009}
F.~{Kerber} and A.~{van der Schaft}, ``Assume-guarantee reasoning for linear
  dynamical systems,'' in {\em Proceedings of the European Control Conference},
  pp.~5015--5020, 2009.

\bibitem{kerber2010}
F.~Kerber and A.~van~der Schaft, ``Compositional analysis for linear systems,''
  {\em Systems \& Control Letters}, vol.~59, no.~10, pp.~645--653, 2010.

\bibitem{kerber2011}
F.~{Kerber} and A.~J. {van der Schaft}, ``Decentralized control using
  compositional analysis techniques,'' in {\em Proceedings of the IEEE
  Conference on Decision and Control and European Control Conference},
  pp.~2699--2704, 2011.

\bibitem{willems2007a}
J.~C. {Willems}, ``The behavioral approach to open and interconnected
  systems,'' {\em IEEE Control Systems Magazine}, vol.~27, no.~6, pp.~46--99,
  2007.

\bibitem{willems1983}
J.~C. Willems, ``Input-output and state-space representations of
  finite-dimensional linear time-invariant systems,'' {\em Linear Algebra and
  its Applications}, vol.~50, pp.~581--608, 1983.

\bibitem{polderman2000}
J.~W. Polderman, ``A new and simple proof of the equivalence theorem for
  behaviors,'' {\em Systems \& Control Letters}, vol.~41, no.~3, pp.~223--224,
  2000.

\bibitem{kaczorek2007}
T.~Kaczorek, {\em Polynomial and Rational Matrices}.
\newblock Springer-Verlag London, 2007.

\end{thebibliography}
\end{document}